\documentclass{article}
\usepackage{amssymb,amsfonts,amsmath,amsthm,amscd}
\usepackage{makeidx}
\usepackage[english]{babel}
\makeindex
\usepackage[x11names]{xcolor}
\usepackage{tikz-cd}  

\usepackage{hyperref}
\input xy
\xyoption{all}
\usepackage[all]{xy}

\usepackage{fancyhdr}
\usepackage{makeidx}
\makeindex %indice alfab�tico
\usepackage{vmargin} %para el margen
\setmargins{3.3cm}       % margen izquierdo
{3cm}                        % margen superior
{14.5cm}                      % anchura del texto
{23cm}                    % altura del texto
{10pt}                           % altura de los encabezados
{0.5cm}                           % espacio entre el texto y los encabezados
{0pt}                             % altura del pie de p�gina
{1cm}                           % espacio entre el texto y el pie de p�gina

\numberwithin{equation}{section}

\theoremstyle{definition}
\newtheorem{teo}{Theorem}[section]
\newtheorem{pro}[teo]{Proposition}

\newtheorem{defi}[teo]{Definition}

\newtheorem{lem}[teo]{Lemma}
\newtheorem{ex}[teo]{Example}

\newtheorem{rem}[teo]{Remark}
\newtheorem{coro}[teo]{Corollary}

\newtheorem{exe}[teo]{Example}
\newcommand{\m}{{}^{-1}}

\newcommand{\Z}{\mathbb{Z}}
\newcommand{\X}{\mathcal{X}}

\newcommand{\G}{{\mathcal G}}

\title{\textbf{Partial actions  on quotient spaces  and globalization }}
\author{
	Luis Mart\'inez, H\'{e}ctor Pinedo and  Andr\'{e}s Villamizar\\
	\small  Escuela de Matem\'{a}ticas\\
	\small Universidad Industrial de Santander\\
	\small Cra. 27 calle 9, Bucaramanga, Colombia\\
	\small  e-mail: luchomartinez9816@hotmail.com, hpinedot@uis.edu.co, andresvillamizar1793@hotmail.com\\
}
\date{\today}

\begin{document}

	\maketitle
	\begin{abstract}Given a partial action of a topological group $G$ on a space $X$ we determine properties $\mathcal P$ which can be extended from $X$ to its globalization. We treat the cases when  $\mathcal P$ is any of the following:  Hausdorff, regular, metrizable,  second countable and having invariant metric.   Further, for    a normal subgroup $H$  we introduce  and study a  partial action  of $G/H$ on the orbit space $X/\!\sim,$  applications to   invariant metrics and inverse limits are presented.
		
	\end{abstract}

	\section{Introduction}   Given an action $a : G \times Y \to Y$ of a group $G$ on a space $Y$ and an invariant subset $X$ of $Y$ (i.e. $a(g, x) \in X, x\in X,\,\, g\in G$), the restriction of $a$ to $G \times X$ is an action of $G$ over $X.$ If $X$ is not invariant, we get what is called a {\it partial action} on $X,$ that is a collection of partial maps $\eta_g, g\in G$ on $X$ satisfying $\eta_1 = {\rm id}_X$ and $\eta_g \circ \eta_h \subseteq \eta_{gh}$, for each $g, h \in G$. The notion of
partial group action  appeared in the context of $C^*$-algebras
 in  \cite{Ruy1}, there $C^*$-algebraic crossed products by partial automorphisms played an important role to analyze and characterize  their internal structure.    Since \cite{Ruy1},  partial actions have been spreading in several branches of mathematics,  for a detailed account  on partial actions the interested reader may consult \cite{DO} or \cite{Ruy}.
 A relevant question is if  a partial action can be obtained by restriction of a corresponding collection of maps on some superspace. In the topological context, this  is known as the globalization problem  and was studied in \cite{AB} and independently in \cite{KL}. It was proven that for any partial action $\eta$ of a topological group $G$ on a topological space $X$ there is a topological superspace $Y$ of $X$ and a continuous action $\mu$ of $G$ on $Y$ such that  the restriction of  $\mu$ to $X$ is $\eta$.  Such a space  is called a globalization of $X.$ It is also shown that there is a minimal globalization $X_G$ called the enveloping space of $X.$

We shall mainly work with partial actions for which the partial maps have clopen domains, that is closed and open, this kind of partial actions were  considered  in \cite{DEX} where the authors studied  the ideal structure of the algebraic partial crossed product  $\mathcal{L}_c(X)\rtimes G$ being $\mathcal{L}_c(X)$ the algebra consisting of all locally constant, compactly supported functions on $X,$ while in \cite{Ruy2} the authors showed that partial  actions on the Cantor set by clopen subsets are exactly the ones for which the enveloping space is Hausdorff, also in \cite{BE} partial actions with clopen domains were relevant  to introduce and study topological entropy for a partial action of $\Z$ on metric spaces, and in \cite{GOR} the authors studied topological dynamics arising from partial actions on clopen subsets of a compact space.

Our work is organized in the following way: After the introduction, in Section \ref{pactions} we present some  notions, examples  and results that will be useful during the work,  especially Proposition \ref{cert1} gives conditions for the enveloping space to be $T_1$ while Theorem \ref{globor} establish that the globalization  of a partial action is actually  an orbit space. At the beginning of  Section \ref{parprof}  we treat the question if  a structural  property $\mathcal P$ of a space $X$ endowed with a partial action of a group $G$ is  inherited by  the spaces $X/\!\sim_G$ and $X_G$ (see equations \eqref{porbit} and \eqref{eqgl} for the proper definitions of  $X/\!\sim_G$ and $X_G,$ respectively). To do so we first show  in Lemma \ref{cerrada}  that the quotient map $\pi_G$ defined in \eqref{qmap} is perfect, this  allows us to present  in  Theorem \ref{propiedades} sufficient conditions in which an   affirmative answer holds for when $\mathcal P$ is any properties of being  Hausdorff, regular, metrizable and second countable. Second part of Section \ref{parprof}  deals with invariant metrics, there we  give in  Theorem \ref{metrinv}  a condition for a space $X$ with a partial action of a compact group so that it admits an invariant metric. At this point  it is important  to note  that in the classical case of finding characterizations of $G$-spaces having  invariant metric  have been extensively studied, in particular it is known that if a space $X$ with a global action admits an invariant metric, then the orbit space  $X/\!\sim_G$ is metrizable provided that is $T_1,$ however, this result does not hold for partial actions, where one needs to impose regularity conditions (see  Remark \ref{nomet} and  Proposition \ref{orbmet}, respectively).  In Section \ref{parinq} we  take a partial action $\eta$ of $G$ on a space $X$, a normal subgroup $H$ of $G$ and we show  in Theorem \ref{lem4.3} how to construct a partial action $\eta_{G/H}$ of $G/H$ on the orbit space $X/\sim_H,$ moreover, in the same Theorem is shown that the orbit spaces $(X/\sim_H)/\sim_{G/H}$ and $X/\!\sim_G$  are homeomorphic. The structure of the partial action $\eta_{G/H}$ as well as its globalization are presented in Theorem \ref{pro4.4}. As an application for the construction of  $\eta_{G/H}$  we treat in Proposition \ref{parinv}  partial actions on inverse limits, where  we provide suitable conditions for which a space $X$ is $G$-equivalent to an inverse limit $\displaystyle{\lim_{\longleftarrow}}X_i,$ and such that the partial action on $X$ satisfies a compatibility relation with the partial actions associated to $X_i$

Throughout the work, several examples are shown to clarify the notions and results.

\section{Preliminaries}\label{pactions} 
Let $G$ be a group  with identity element $1,$ $X$ be a  set, and   $\eta: G\times X\to X$, $(g,x)\mapsto g\cdot x$  be a partially defined  function, that is, a function whose domain, denoted by $G*X,$ is contained in $ G\times X.$ We shall write $\exists g \cdot x$ to mean that $(g, x)$ belongs to $G*X.$ We say that  $\eta$ is a {\it  partial action}  of $G$ on $X$ if for each $g,h\in G$ and $x\in X$ the following assertions hold:
	\begin{enumerate}
		\item [(PA1)] If $\exists g\cdot x$, then $\exists g^{-1}\cdot (g\cdot x)$ and $g^{-1}\cdot (g\cdot x)=x$,
		\item [(PA2)] If $\exists g\cdot(h\cdot x)$, then $\exists (gh)\cdot x$ and $g\cdot(h\cdot x)=(gh)\cdot x$,
		\item [(PA3)] $\exists 1\cdot x$ and $1\cdot x=x.$
	\end{enumerate}
 For $g\in G$ set  $X_g=\{x\in X\mid \exists g\m\cdot x \}.$
%\begin{itemize} 
%\item 
%\item  
%\item   $G*X=\{(g,u)\in G\times X\mid \exists g\cdot x\}$ is  the domain of $\eta.$
%\end{itemize}	
Then $\eta$ induces a family of bijections $\{\eta_g\colon X_{g\m}\ni x\mapsto g\cdot x\in X_g \}_{g\in G}.$   We also denote this family  by $\eta.$ Notice that $\eta$ {\it acts} (globally) on $X$ if $\exists g\cdot x,$ for all $(g,x)\in G\times X,$ or equivalently, $X_g=X,$ for any $g\in G.$
The following result characterizes partial actions in terms of a family of bijections.
\begin{pro} \label{QR}{\rm  \cite[Lemma 1.2]{QR} \label{fam} A partial action $\eta$ of $G$ on $X$ is a family $\eta=\{\eta_g\colon X_{g\m}\to X_g\}_{g\in G},$ where $X_g\subseteq X,$  $\eta_g\colon X_{g\m}\to X_g$ is bijective, for  all $g\in G,$  and:
\begin{itemize}
\item[(i)]$X_1=X$ and $\eta_1={\rm id}_X;$
\item[(ii)]  $\eta_g( X_{g\m}\cap X_h)=X_g\cap X_{gh};$
\item[(iii)] $\eta_g\eta_h\colon X_{h\m}\cap  X_{ h\m g\m}\to X_g\cap X_{gh},$ and $\eta_g\eta_h=\eta_{gh}$ in $ X_{h\m}\cap  X_{ h\m g\m};$
\end{itemize}
for all $g,h\in G.$}
\end{pro}

%\textcolor{blue}{Now we will define the {\it topological partial actions.}}

\begin{defi}
	Let $G$ be a topological group and $X$ be a topological space. A topological partial action of $G$ on $X$ is a partial action $\eta=\{\eta_g\colon X_{g\m}\to X_g\}_{g\in G}$ on the underlying set $X$ such that $X_g$ is open, $\eta_g$ is homeomorphism for any $g\in G.$ Moreover we say that $\eta$ is continuous if   $\eta:G*X\rightarrow X$ is continuous, where $G\times X$ has the product topology and  $G*X$ is endowed  with the relative topology.
\end{defi}

\textit{ Throughout this paper $G$ will denote a Hausdorff topological group, $X$ a topological space and all partial actions will be topological.}\\

%{\it Throughout the rest of this paper $G$ will denote a topological group, $X$ a topological space,  the space $G\times X$ has the product topology and $G*X$ is endowed  with the  topology of subspace, moreover  all partial actions will be continuous and topological, that is if every $X_g$ is open and $\eta_g$ is a homeomorphism, $g\in G.$}
%\medskip
%From now on in this work $G$ will denote a topological group and $X$ a topological space,. Moreover $\eta: G*X\to X$  will denote a partial action.  Then $\eta$ is called   {\it topological } 
%\begin{enumerate}
%\item 
%\item We say that $\eta$ is a$ if  $\eta$ is continuous, we say that $\eta$
%is a {\it continuous partial action}.

Now we present an example of a continuous and topological partial action that will be useful in  Section \ref{parprof}. We endow $\mathbb{Z}$ with the {\it $p$-adic topology} $\mathcal{T}_p$, where $p$ is a prime number. For the reader's convenience we recall its construction here. See \cite [Example 1.18]{T} for details. The family $\mathcal{V}=\{p^k\mathbb{Z}\}_{k\in \Z^+}$  satisfy  the conditions given in \cite [Theorem 1.13]{T}, then  $\mathcal{B}=\{m+ p^k\mathbb{Z}: m\in\mathbb{Z},\ k\in \Z^+\}$ is a basis for the topology $\mathcal{T}_p$ and   $(\mathbb{Z},+, \mathcal{T}_p)$ is  a topological group.

\begin{exe}\label{padtop}  Let $X$ be a disconnected topological space, $U\subseteq X$ be a proper clopen set, $f:U\rightarrow U$ be a  homeomorphism and $n\in \Z$.  We set $f^0={\rm id}_U,$ if $n\in \Z^+$  we write $f^n$ as $n$-times the composition of $f$  with itself and if $n<0,$ then $f^{n}=(f^{-1})^{-n}.$ We define a  partial action {$\eta$} of $\Z$ on $X$ by setting 
	\begin{equation}\label{zx}\mathbb{Z}*X=(\mathbb{Z}\times U)\cup (\{0\}\times X)\,\,\, {\rm and}\,\,\, \eta: \mathbb{Z}*X\ni (n,a)\mapsto \left\lbrace\begin{array}{c} f^n(a),~if~a\in U, \\ a,~if~n=0~and~a\not\in U. \end{array}\right.\in X.
	\end{equation}
	 Suppose there is $p$ a prime number such that   $f^p={\rm id}_U,$   and consider  $\mathbb{Z}$ with the $p$-adic topology.  Since $U$ is open, then  $\eta$ is  a topological  partial action.  To show that it is continuous take $(n,x)\in \mathbb{Z}*X$, and  $V\subseteq X$ an open set such that $\eta(n,x)\in V$. There are two cases to consider.

	{\bf Case 1:}  $x\in U$.  Then $\eta(n,x)=f^n(x)\in V$. Since $V\cap U$  is open in $U$,  there is $Z\subseteq U$ an open set such that  $f^n(Z)\subseteq V\cap U$.  First, suppose that  $p$ does not divide $|n|$. Then the open set  $W=[(n+Z_1)\times Z]\cap (\mathbb{Z}*X)\subseteq \mathbb{Z}*X$ satisfies $\eta(W)\subseteq V,$ because for $(t,y)\in W$ we have  $y\in Z\subseteq U$ and there is  $m\in \mathbb{Z}$ such that $t=n+pm$. Note that  $t\neq 0$  since $p$ does not divide  $|n|$.  Further since $y\in U$, we get  $(n,y)\in \mathbb{Z}*X,$ and $(pm,f^n(y))\in \mathbb{Z}*X.$ The fact $f^n(y)\in U$ gives $$\eta(t,y)=f^t(y)=f^{pm}(f^n(y))=f^n(y)\in V.$$
	
Now if   $p$ divides $|n|$ we  let  $i=\max\{k\in \mathbb{Z}^+: p^k\ {\text divides}\  n\}$. Consider the open set   $W=[(n+Z_{i+1})\times Z]\cap (\mathbb{Z}*X)$. Then for  $(t,y)\in W$, by the maximality of $i$  there is  $m\in\mathbb{Z}$ such that $t=n+p^{i+1}m$, $y\in Z\subseteq U$ and  $t\neq 0.$ Since $y\in U$,  we get  $$\eta(t,y)=f^{n+p^{i+1}m}(y)=f^{p^{i+1}m}(f^n(y))=f^n(y)\in V,$$ and $\eta(W)\subseteq V$.

	{\bf Case 2:} $x\notin U.$  By \eqref{zx} we have   $n=0$ and $\eta(n,x)=x\in V$.  Since  $U$ is closed, then   $V\cap (X\setminus U)$ is an open subset of  $X$ containing $x$.  Take  $Z\subseteq X$  open such that $x\in Z\subseteq V\cap (X\setminus U)$ and let  $W=(Z_1\times Z)\cap (\mathbb{Z}*X).$ It is clear that  $(n,x)=(0,x)\in W$. Further if  $(t,y)\in W$, then $y\notin U$ and  $t=0$  from this we get  $\eta(t,y)=\eta(0,y)=y\in V,$ showing that $\eta$ is continuous.
\end{exe}
\subsection{On the enveloping space} Partial actions can be induced from global ones as the following example shows.

\begin{ex}  \label{indu} (Induced partial action)
Let $u \colon G\times Y\to Y$ be a continuous action of $G$ on a topological space $Y$ and $X\subseteq Y$ an open set.  For $g\in G,$ set $X_g=X\cap u_g(X)\,\,\, \text{ and let}\,\,\, \eta_g=u_g\restriction X_{g\m}.$
 Then $\eta \colon G* X\ni (g,x)\mapsto \eta_g(x)\in X $ is a continuous and  topological partial action of $G$ on $X.$  In this case we say that $\eta$ is  {\it induced} by  $u.$
\end{ex}
\begin{rem}\label{domiclopen}
	Given a continuous global action $\eta$ of  $G$ on $X$,  its induced partial action on an  open (resp. closed) subset  $Y$ of $X$ has open (resp. closed) domain in $G\times Y$.
\end{rem}

An important problem on  partial actions is whether they can be induced by  global actions. In the topological sense, this turns out to be affirmative and a proof was presented in  \cite[Theorem 1.1]{AB} and independently in \cite[Section 3.1]{KL}.  For the reader's convenience, we recall their construction. Let $\eta$ be a partial action of $G$ on $X.$ Define an
equivalence relation on $ G\times X$ as follows:
\begin{equation}
\label{eqgl}
(g,x)R(h,y) \Longleftrightarrow x\in X_{g\m h}\,\,\,\, \text{and}\, \,\,\, \eta_{h\m g}(x)=y,
\end{equation}
and denote by  $[g,x]$   the equivalence class of the pair $(g,x).$
Consider $X_G=(G\times X)/R$  with  the quotient topology and the map \begin{equation}
	\label{action}
	\mu \colon G\times X_G\ni (g,[h,x])\mapsto [gh,x]\in X_G,
\end{equation}
is a well defined action, and the map  \begin{equation}
	\label{iota}
	\iota \colon X\ni x\mapsto [1,x]\in X_G,
\end{equation} is injective.

\begin{defi}\label{globa} Let $\eta$ be a  partial action of $G$ on $X.$ The action $\mu$ defined in  \eqref{action} is called the enveloping action of $\eta$   and $X_G$ is the
	enveloping space  or globalization of $X.$ \end{defi}
In the next result we summarize some basic results about the enveloping space and the enveloping action, see  \cite[Theorem 1.1]{AB},  \cite[Theorem 3.13]{KL} and \cite[Proposition 3.9]{KL}.
\begin{pro} \label{pregunta} Let $\eta$ be a partial action of $G$ on $X.$ Then the following assertions hold.
	\begin{itemize}
		\item [(i)]  The maps $\mu$ and $\iota$ are continuous.
		\item  [(ii)]  If $\eta$ is continuous and $G*X$ is open in $G\times X,$ then $\iota$ is an open map. % such that $G\cdot \iota(X)=X_G,$
	\item [(iii)] The quotient map 
	\begin{equation}\label{quo}  q:  G\times X\ni(g,x)\mapsto [g,x]\in X_G,
	\end{equation}
	is continuous and open.
	\end{itemize}

\end{pro}

 %and if $G\ast X$ is open then $\iota$ is open.	

 %and the pair $(\mu, X_G)$ is called  the globalization of $X.$

Now we provide conditions for $X_G$ to be $T_1$. 

\begin{pro}\label{cert1} Let $\eta$ be a continuous partial action of $G$ on $X.$ Consider the following assertions.
\begin{itemize}
\item [(i)] $G*X$ is closed;
\item [(ii)] For any $x\in X$ the set $G^x=\{g\in G\mid \exists g\cdot x\}$   is closed;
\item [(iii)] $X_G$ is $T_1.$
\end{itemize}
Then (i)$\Rightarrow$ (ii), and   (ii)$ \Rightarrow $(iii) provided that $X$ is Hausdorff.
\end{pro}
\begin{proof}
  (i)$\Rightarrow$ (ii)  For a net $(g_\lambda)_{\lambda\in\Lambda}$ in $G^x$ such that  $\lim g_\lambda=g$,  one has $(g_\lambda,x)_{\lambda\in\Lambda}\to (g,x)\in \overline{G*X}=G*X$, thus $g\in G^x$ and  $G^x$ is closed.\\
    For the rest of the proof we assume that $X$ is Hausdorff.

\noindent (ii)$\Rightarrow$ (iii)  Take  $(g,x)\in G\times X,$ and let   $q$  be the quotient map defined in \eqref{quo}, then 
$$q\m(\{[g,x]\})=\bigcup_{l\in G}\{gl\m\}\times \eta_l( \{x\}\cap X_{l\m})=\{(gl\m, l\cdot x)\mid l\in G^x\}.$$ 
We prove that $q\m(\{[g,x]\})$ is closed. For this let $(h,y)\in \overline {q\m(\{[g,x]\}) },$ then there exists a net  $\{l_i\}_{i\in I}$ in $G^x$ such that $(gl\m_i, l_i\cdot x)\to (h,y), $ in particular $l_i\to h\m g\in G^x.$ Set $\eta^x: G^x\ni g\mapsto g\cdot x\in X,$  using the fact that  $\eta^x$ is continuous one gets 
$l_i\cdot x\to (h\m g)\cdot x,$ and $y=(h\m g)\cdot x$ because of the uniqueness
of limits in Hausdorff spaces. 
 From this we obtain $$(h,y)=(g(h\m g)\m,(h\m g)\cdot x )\in q\m(\{[g,x]\}),$$  thus  $X_G$ is $T_1.$ %To show (iii)$\Rightarrow$ (ii) suppose  $X_G$ is $T_1$ then $\pi\m(\{[1,x]\})=\{(l\m, l\cdot x)\mid l\in G^x\}$ is a subset of $G*X$ which is closed in $G\times X.$ Now  the closed map  $\rho$ defined in \eqref{ro} satisfies $\rho[\pi\m(\{[1,x]\})]=G^x\times\{x\},$ and thus  the set $G^x\times\{x\}$ is closed, finally  since $\{x\}$ is compact we get that $G^x$ is closed, as desired.
 \end{proof}
\begin{rem}\label{notrec}With respect to Proposition \ref{cert1} we have the next.
	  \begin{itemize}
		\item The space  $X_G$ is $T_1$ when $G$ is discrete and $X$ is Hausdorff.
		\item  Part (ii) $\Rightarrow$ (i)  does not necessarily holds. 
		Indeed, for  the partial action  of $\Z_2=\{1,-1\}$ on $X=[0,1]$ presented in \cite[Example 1.4.]{AB} that is $\alpha_1={\rm id}_X$ and $\alpha_{-1}={\rm id}_V,$ where $V=(0,1].$ One has that  $\Z_2^x$ is closed for any $x\in [0,1]$ while $\Z_2*[0,1]=\{(-1,0)\} \cup( \Z_2\times V)$ is not closed in $\Z_2\times [0,1].$
		\item Also part (iii) $\Rightarrow$ (ii) does not hold in general, for this let $G=GL(2;\mathbb{R})$ be the  general linear group of degree $2$ acting  partially   on $\mathbb{R}$ as follows: For  $g=	\begin{pmatrix}
			a & b \\
			c & d 
		\end{pmatrix}\in G$, set $\mathbb R_{g^{-1}}=\{x\in \mathbb{R}: cx+d\neq 0\}$ and  $\eta_g:\mathbb R_{g^{-1}}\ni x\mapsto \displaystyle\frac{ax+b}{cx+d} \in \mathbb R_g$. There is a homeomorphism from $\mathbb R_G$ to  the space $\mathbb C$ of complex numbers, then   $\mathbb R_G$ is  Hausdorff but $G^0=\left\lbrace\begin{pmatrix}
			a & b \\
			c & d 
		\end{pmatrix}\in G: d\neq 0\right\rbrace$ is not closed in $G.$ 
	\end{itemize} 
\end{rem}

\begin{defi} Suppose that the spaces $X$ and $Y$ are equipped with  partial actions $\eta$ and $\rho$ by $G.$ A function $\epsilon: X \to Y$ is called a or $G$-map if for every $(g,x)\in  G*X$  then   $(g,\epsilon(x))\in G*Y$ and $\epsilon \eta(g,x) = \rho (g,\epsilon(x)).$ If moreover $\epsilon$ is a homeomorphism we say that $X$ and $Y$ are $G$-equivalent.
\end{defi}

We have the next.

\begin{pro} \label{kellen}The following assertions hold. 
\begin{enumerate}
\item  [(i)] Let $X$ and $Y$ be two $G$-equivalent spaces. Then    $X_G$ and $Y_G$ are homeomorphic, as well as $G*X$ and $G*Y.$
\item [(ii)]  Let $\beta:G\times Y\rightarrow Y$ be a continuous action of $G$  on a space $Y.$ Let    $X\subseteq Y$  be an open set such that  $G\cdot X= Y$ and  $\eta:G*X\rightarrow X$ be the induced partial action of $\beta$ on $X$ (see Example \ref{indu}). Then the spaces  $X_G$ and  $Y$ are  $G$-equivalent.
\end{enumerate}\end{pro}

\begin{proof} Part (i) is clear, for (ii) let $i:G\times X\to G\times Y$ be the inclusion and  $\alpha:X_G\ni [g,x]\mapsto \beta(g,x)\in Y$, then the following diagram 
		\begin{center}
		$\xymatrix{G\times X \ar[d]_-{q} \ar[r]^-{i}& G\times Y\ar[d]^-{\beta}\\
			X_G\ar[r]_-{\alpha}& Y
		}$
	\end{center}
is commutative. Moreover, by \cite[Proposition 3.5]{KL} the map $\alpha$ is a well defined bijection, moreover it is continuous because  the map $\alpha\circ q$ is continuous. Also, since  $\beta$ is open the map  $\alpha$ is a homeomorphism, finally the fact that it  is a $G$-map is straighforward.
\end{proof}
\subsection{The orbit equivalence relation}
Given a  partial action   $\eta$ of $G$ on $X$  the {\it orbit equivalence relation}  $\sim_G$ on $X$ is:
\begin{equation}\label{porbit}x\sim_G y\Longleftrightarrow\,\, \exists g\in G^x  \,\,{\rm such \,\,that}\,\, g\cdot x=y,\end{equation}
for each $x,y\in X$. The elements of $X/\!\sim_G$ are the {\it orbits} $G^x\cdot x$ with  $x\in X$  and $X/\!\sim_G$ is endowed with the quotient topology. By \cite[Lemma 3.2]{PU1} the {\it induced quotient map of $\eta$} 
\begin{equation}\label{qmap}\pi_G: X\ni x\mapsto G^x\cdot x \in X/\!\sim_G,\end{equation} is continuous and open. %Moreover it is not difficult to see that 
%\begin{equation}\label{impim}\pi\m_G( \pi_G(V))=\eta((G\times V)\cap (G*X) ),\end{equation} for any subset $V$ of $X.$
%Item [(ii)] result extendes  \cite[Proposition 3.5]{KL}  to the realm of topological partial actions and will be useful in the sequel.

It is known that  globalizations of topological spaces endowed with a partial action can be seen as orbit equivalence spaces. Indeed the following result was shown in \cite[Theorem 3.3]{PU1}.

\begin{teo} \label{globor}Let  $\eta$ be  a topological partial action   of $G$ on $X,$ then the family $\hat{\eta}=\{\hat{\eta}_g: (G\times X)_{g^{-1}}\rightarrow (G\times X)_g\}_{g\in G},$  where $(G\times X)_g=G\times X_g$ and 
	$$\hat{\eta}_g: G\times X_{g\m}\ni (h,x) \mapsto (hg^{-1},\eta_g(x)) \in G\times X_g,$$ 
	is a topological  partial action of $G$ on $G\times X,$ and the enveloping space   $X_G$ of $\eta$ is the space of orbits  of $G \times X$ by $\hat \eta.$ %\footnote{Sabiendo que la envolvente es un espacio de orbitas, ser\'a posible dar un ejemplo de $X_G$ usando este teorema?}
\end{teo}
Let $\eta$ be a partial action of $G$ on $X,$  and $H$ be a subgroup of $G,$ then  the family $\eta_H: \{\eta_h\colon X_{h\m}\to X_h\}_{h\in H}$ is a partial action of $H$ on $X.$ The corresponding orbit equivalence relation of $\eta_H$ is denoted by $\sim_H$. 

For convenience,  the orbits in the space $X_G/\!\sim_H$ will be denoted by $H[g,x]$ for any $[g,x]\in X_G$. We finish this section with the next.

\begin{lem}\label{lemaaux}
	Let $\eta$ be a continuous partial action of $G$ on $X$ with $G*X$ open. Then for each subgroup $H$ of $G$ the map 
	\begin{equation}\label{var}\varphi:X/\!\sim_H \ni H^x\cdot x\mapsto H[1,x]\in X_G/\sim_H \end{equation} is an embedding, that is  continuous, open and injective. \end{lem}
\begin{proof}
	First of all note that $\varphi$ is well defined. In fact, let $x,y\in X$ be such that $x\sim_H y$ and take  $h\in H^x$ with $\eta_h(x)=y$. Thus, 
	$[1,y]\stackrel{\eqref{eqgl}}=[h,x]\stackrel{\eqref{action}}=\mu_h([1,x])$
	and $[1,y]\sim_H[1,x],$ then $\varphi$ is well defined. It is easy to check that $\varphi$ is injective. To prove that $\varphi$ is continuous, consider $\pi_H: X\rightarrow X/\!\sim_H$ and $\Pi_H:X_G\rightarrow X_G/\!\sim_H$ the  corresponding quotient maps. Since the map $\iota$ defined in \eqref{iota}  is continuous and   $\varphi\circ \pi_H=\Pi_H\circ \iota$ we conclude that $\varphi$ is continuous. It remains to check that $\varphi$ is open, let $U\subseteq X/\sim_H$ be an  open set, then $\varphi(U)=\Pi_H(\iota(\pi_H^{-1}(U)))$  is open because  $\pi_H^{-1}(U)$ is open in $X$ and the functions $\iota$ and $\Pi_H$ are open thanks to Proposition \ref{pregunta} and  \cite[Lemma 3.2]{PU1}, respectively. Therefore $\varphi$ is an open map.
\end{proof}
\section{Properties preserved by the envelo\-ping space and invariant metrics}\label{parprof}
Recall that a continuous surjection $f : X \rightarrow Y$ is {\it perfect} if it is closed  and $f\m(\{y\})$ is compact for all $y \in Y.$

%\textit{Throughout the rest of this paper $G$ will denote a topological group $X$ a topological space and all partial actions will be continuous and topological.}\medskip

We proceed with the next.

\begin{lem}\label{cerrada} Let  $\eta:G*X\rightarrow X$ be a  continuous partial action such that  $G*X$ is closed in $G\times X$ and  $G$ is compact, then the following assertions hold.
\begin{enumerate}
\item [(i)]  $\eta$ is closed;
\item [(ii)] The maps $\pi_G$ and $\hat{\pi}_G$  are perfect, being  $\hat{\pi}_G$ the corresponding quotient map of $\hat\eta$ in Theorem \ref{globor} .
%\item %If $X$ is Hausdorff, then  $X/\sim_G$  and  $X_G$  are Hausdorff spaces.

\end{enumerate}
\end{lem}
\begin{proof}
(i). Let  $C$ be a nonempty closed subset of $G*X$   and $y\in \overline{\eta(C)},$ then there is a directed set $\Lambda$ and a net $(g_\lambda,x_\lambda)_{\lambda\in \Lambda}$ in $C$ such that $\lim g_\lambda\cdot x_\lambda =y$. Since $G$ is compact, we can suppose that $\lim g_\lambda =g$, for some $g\in G$. Note that $(g_{\lambda}^{-1}, g_{\lambda}\cdot x_{\lambda})_{\lambda\in \Lambda}$ is a net in $G*X$ and $\lim (g_{\lambda}^{-1}, g_{\lambda}\cdot x_{\lambda})= (g^{-1},y),$ then $(g^{-1},y)\in G*X$ because this is a  closed subset of $G\times X$. Now consider the net $(g_\lambda,x_\lambda)_{\lambda\in \Lambda}=(g_\lambda, g_{\lambda}^{-1}\cdot( g_{\lambda}\cdot x_{\lambda}))_{\lambda\in \Lambda}$ in $C,$ then 
$$(g,g^{-1}\cdot y)=\lim (g_\lambda, g_{\lambda}^{-1}\cdot( g_{\lambda}\cdot x_{\lambda}))=\lim (g_\lambda,x_\lambda)\in C ,$$  and $y=g\cdot (g^{-1}\cdot y)=\eta(g,g^{-1}\cdot y)\in \eta(C)$ which implies that $\eta$ is closed.

\noindent (ii) The map  $\pi_G$ is closed because of (i) above and  the equality $\pi\m_G( \pi_G(F))=\eta((G\times F)\cap (G*X) ),$ for any closed subset $F$ of $X.$ Hence to prove our assertion we need to check that $\pi_G^{-1}(\pi_G(x))$ is a compact for any $x\in X$. First, by Proposition \ref{cert1} we have that $G^x$ is a  compact subset $G$,  then $\pi_G^{-1}(\pi_G(x))=G^x\cdot x=\eta(G^x\times \{x\})$ is a compact subset of $X$. To show that $\hat{\pi}_G$ is closed we have  by 
\cite[Proposition 2.6]{MPV} that the map  $\hat{\eta}$ is continuous, moreover from \cite[Corollary 3.3]{MPV} we get that $G*(G\times X)$ is closed in  $G\times (G\times X)$ then  the result follows. 
%{\bf Case 2. $X$ is compact and Hausdorff.} Let $F$ be a closed subset of $X,$ by  \eqref{impim} we know that 
\end{proof}
 %Let $x,y\in X$ such that $\pi_G(x)\neq\pi_G(y)$. By 1. above the sets $G^x\cdot x=\pi_G^{-1}(\pi_G(x))$ and  $G^y\cdot y=\pi_G^{-1}(\pi_G(y))$ are  disjoint compact subsets of $X$,  since $X$ is Hausdorff there are  open disjoint subsets $U$ and  $V$ of $X$ such that $G^x\cdot x\subseteq U$ and $G^y\cdot y\subseteq V$. Since $\overline{U}\cap V=\emptyset$, we obtain  $\overline{U}\cap G^y\cdot y=\overline{U}\cap \pi_G^{-1}(\pi_G(y))=\emptyset$, and  $\pi_G(y)\notin \pi_G(\overline{U})$. Finally the fact that $\pi_G$ is a closed map implies that  $X/\sim_G\setminus \pi_G(\overline{U})$ and  $\pi_G(U)$  are disjoint neighbourhoods of  $\pi_G(y)$ and $\pi_G(x)$, respectively,  and we conclude that $X/\sim_G$ is Hausdorff. To show that $X_G$ is Hausdorff consider the topological partial action $\hat \eta$ of  Theorem \ref{globor},  by 
   %\cite[Proposition 2.6]{MPV} we have that $\hat{\eta}$ is continuous, moreover from \cite[Corollary 3.3]{MPV} we get that $G*(G\times X)$ is closed in  $G\times (G\times X)$ and the result follows  from Theorem  \ref{globor}.

\begin{teo}\label{propiedades}
Let $G$ be a compact group and $\eta:G*X\rightarrow X$ be a  continuous partial action such that  $G*X$ is closed in $G\times X$.  Let $\mathcal{P}$ be any of the properties: Hausdorff, regular, metrizable and second countable. Then the following statements hold.
\begin{enumerate}
 \item [(i)] If $X$ is $\mathcal{P},$ then $X/\!\sim_G$ is $\mathcal{P}.$ 
\item [(ii)] If $G\times X$ is $\mathcal{P},$  then   $X_G$ is $\mathcal{P}.$ 
\end{enumerate}
\end{teo}
\begin{proof} (i)  This follows from item (ii) in  Lemma \ref{cerrada} and \cite[Theorem 5.2]{du} while  (ii)  is a consequence of  item (ii) in  Lemma \ref{cerrada},  item (i) above  and the last assertion in Theorem \ref{globor}.
\end{proof}

%Recal that a space $X$ is Polish if  is separable and  completely metrizable.  In \cite[Theorem 4.7]{PU1} were presented son sufficient conditions for the enveloping space $X_G$ to be a Polish space provided that  $G$ and $X$ are Polish, in this sense the following result is a consequence of 2. in Theorem \ref{propiedades} and  \cite[Theorem 4.1]{PU1} .
%
%\begin{pro} Let $G$ be a compact Polish group and $X$ be a Polish space and   $\eta:G*X\rightarrow X$ be a  partial action such that  $G*X$ is closed in $G\times X$.  Then  $X_G$ is Polish.
%\end{pro}
\begin{rem}\label{cernes} We remark the following facts.
\begin{enumerate}
\item [(i)] In general the assumption that  $G*X$ is closed in $G\times X$ cannot be removed in   part (ii) of Theorem \ref{propiedades}.   Indeed, for  the Abadie's partial action  of $\Z_2=\{1,-1\}$ on $X=[0,1]$ presented in Remark \ref{notrec}  we have by Proposition \ref{cert1}  that the  space $X_{\Z_2}$ is $T_1$ but not Hausdorff. 
\item   [(ii)] Also  the fact that $X_G$ is Hausdorff does not imply that $G$ is compact, for instance in \cite[Proposition 2.1]{Ruy2}  a characterization for $X_G$ to be Hausdorff is presented in the case when $G$ is countable and discrete.
\end{enumerate}
\end{rem}
We illustrate the previous theorem with some examples.
%\begin{exe}
%Una acción de $S^1$. \textcolor{red}{Tomemos un subgrupo finito $H$ de $S^1$ y $W\subseteq \mathbb{C}$ un cerrado que además es cerrado bajo el producto usual de números complejos y $H\subseteq W.$ Si $e\in S^1$ es la identidad, se define $D_e:=\mathbb{C}$ y $\eta_e:=id_{\mathbb{C}.}$ Para cada $g\in H\setminus \lbrace e\rbrace,$ definimos $D_g:=W,$ y $\eta_g:D_{g^{-1}}\rightarrow D_g; x\mapsto gx.$ Si $g\neq e\in S^1\setminus H,$ entonces $D_g=\emptyset$ y $\eta_g$ es la función vacía. Las colecciones $\lbrace D_g\rbrace_{g\in S^1}$ y $\lbrace \eta_g\rbrace_{g\in S^1}$ definen una acción parcial continua $\eta,$ donde $S^1\ast\mathbb{C}=(\lbrace e\rbrace\times\mathbb{C})\cup (H\times W),$ es cerrado.}
%\end{exe}
	\begin{exe}\label{ejeisometr1}
	Consider $X=\mathbb{R}\setminus \{0\}$ as a subspace of   $\mathbb{R}$. A partial action of $\Z_3$ on $X$ is defined as follows. Let  $X_1=(-\infty,0)$ and  $X_2=(0,\infty),$  note that $X_1$ and $X_2$ are clopen subsets of $X$ such that $X=X_1\cup X_2$. Set  $\eta_0={\rm id}_X,$ $\eta_2: X_1\ni x\mapsto -x\in X_2$ and  $\eta_1=\eta\m_2,$ moreover let $$\mathbb{Z}_3*X=(\{0\}\times X)\cup(\{1\}\times X_2)\cup (\{2\}\times X_1),$$  Then  $\eta: \mathbb{Z}_3*X\rightarrow X$,   is  a partial action of $ \mathbb{Z}_3$ on  $X$ such that  $\mathbb{Z}_3*X$ is clopen in  $\mathbb{Z}_3\times X$ thus by  Theorem \ref{propiedades}  the enveloping space $X_{\mathbb{Z}_3}$ is metrizable.
\end{exe}
\begin{exe} \label{ejeisometr2}
	Let $X$ be a disconnected space and $U\subseteq X$ be a non-empty clopen subset of $X$ with $U\neq X.$  Then $\eta: \Z_2* X\rightarrow X$ is a partial action of $\Z_2$ on $X$ where  $\mathbb{Z}_2* X=(\{0\}\times X)\cup (\{1\}\times U)$, and $\eta(1,u)=u$ for any $u\in U $. Since  $\mathbb{Z}_2* X$  is closed in $\mathbb{Z}_2\times X$ we conclude that $X_{\mathbb{Z}_2}$ is metrizable.
\end{exe}

In view of (ii) in Remark \ref{cernes} we give  the next.
\begin{pro}\label{caracteri} Let $G$ be a compact group, $X$ be a compact Hausdorff  space  and 
	$\eta: G*X\rightarrow X$ be a partial action. If $X_G$ is Hausdorff, then  $G*X$ is closed.
\end{pro}
\begin{proof}
	Let  $\{(g_\lambda,x_\lambda)\}_{\lambda\in \Lambda}$  be a net in $G*X$ such that $\lim (g_\lambda,x_\lambda)=(g,x)\in G\times X$. Since $X_G$ is Hausdorff, we have by \cite[Proposition 1.2]{AB} that the space Graph$(\eta)=\{(g,x,y)\in G\times X\times X: (g,x)\in G*X,\ g\cdot x=y\}$ is a closed subset of  $G\times X\times X$, and thus compact. Therefore  we may assume that $(g_\lambda,x_\lambda,g_\lambda\cdot x_\lambda)_{\lambda\in\Lambda}$ converges to  $(g,x,p)\in$ Graph$(\eta)$, for some  $p\in X$. In  particular, $(g,x)\in G*X$ and  $G*X$ is closed.
\end{proof}
Having at hand Proposition  \ref{caracteri} one may ask if  its assumptions imply that if the orbit space $X/\!\sim_G$ is  Hausdorff then  $G*X$ is closed in $G\times X$. But this is not the case as Example \ref{notclo} below shows. % Por el Corolario \ref{esporbit2}, restaría ver si $X/\sim_G$ es Hausdorff, entonces $G*X$ es cerrado. Sin embargo, esto no es cierto como veremos en el siguiente ejemplo.\end{rem}
\begin{exe}\label{notclo}  Consider again the partial action  of $\Z_2$ on $X=[0,1]$ given in \cite[Example 1.4.]{AB}.  We observed in Remark \ref{notrec} that $\mathbb{Z}_2*X$ is not closed in   $\mathbb{Z}_2\times X.$ Moreover,    since  $\eta(1,x)=x$ for any  $x\in (0,1]$ we have  $\pi_{\mathbb{Z}_2}: X\rightarrow X/\!\sim_{\mathbb{Z}_2}$ is injective and thus a  homeomorphism and  $X/\!\sim_{\mathbb{Z}_2}$ is Hausdorff. 
\end{exe}

\subsection{Invariant metrics}
Let  $\eta:G*X\ni (g,x)\mapsto g\cdot x\in  X$ be  a partial action of  $G$ on the metric space $(X,\rho)$. We say that  $\rho$ is {\it $\eta$-invariant} if for any  $g\in G$ and  $x,y\in X_{g^{-1}}$,  $\rho(g\cdot x,g\cdot y)=\rho(x,y)$.
%The spaces endowed with then partial actions considered in Examples  \ref{ejeisometr1}, \ref{ejeisometr2} and    \ref{notclo}, respectively have invariant metrics. For a non trivial example we have the next.
\begin{exe}\label{componer} Let $\eta$ be as in equation \ref{zx}. Suppose that $X$ is metric, $U$ is a clopen subset of $X$ and   $f$ is an isometry, then $\eta$ is a topological and continuous partial action with invariant metric in any of the following cases.
	\begin{itemize}
		\item $\Z$ is considered as a discrete space.
		\item  $\Z$ is endowed with the $p$-adic topology and $f^p={\rm id}_U,$ for some prime number $p.$
	\end{itemize}%With the notation of Example \ref{padtop}
%	Let $U\subseteq X$ be an open set $X$ and  $f:U\rightarrow U$ be a  homeomorphism.  For  $n\in\mathbb{Z}$ we set $f^n={\rm id}_U$ if $n=0,$ $f^n$ is $n$-times the composition of $f$  with itself if $n>0,$ and  $f^{n}=(f^{-1})^{-n}$ if $n<0.$  We define a partial action of $\Z$ on $X$ by setting 
%\begin{equation}\label{zx}\mathbb{Z}*X=(\mathbb{Z}\times U)\cup (\{0\}\times X)\end{equation}  and $\eta: \mathbb{Z}*X\ni (n,a)\mapsto f^n(a)\in X$, 
%Considering $\Z$ with the discrete topology, then  $\eta$ is continuous and  $\mathbb{Z}*X$ is closed in $\mathbb{Z}\times X,$ thus if  
   %and $f$ is an isometry,  then the metric of $X$ is $\eta$-invariant.
\end{exe}
In the context of hyperspaces endowed with partial actions we give the next.
\begin{exe}  Let $\eta:G*X\ni (g,x)\mapsto g\cdot x\in X$ be a continuous partial action of   $G$ on a compact metric space  $(X,d)$. Denote by $2^X$
%$=\{A\subseteq X:\ A\ \text{is\ compact\ and\ nonempty}\}$$ 
the hyperspace of nonempty compact subsets of 
	 $X$ endowed with the Hausdorff metric $d_H,$  which is defined by the rule
$$d_H(A,B)=\inf\{\varepsilon>0: A\subseteq N(B,\varepsilon)\ and\ B\subseteq N(A,\varepsilon)\},$$ where $A,B\in 2^X$ and  $N(A,\varepsilon)=\bigcup\limits_{a\in A} B_d(a;\varepsilon)$. If $\lbrace\eta_g\rbrace_{g\in G}$ is the induced family of bijections by $\eta$,  then follows by \cite[Theorem 3.2]{MPR} that $2^\eta:G*2^X\ni (g,A)\mapsto \eta_g( A)\in 2^X$,  is a continuous partial action of  $G$ in $2^X$, being $$G*2^X=\{(g,A)\in G\times 2^X: (g,a)\in G*X\ (\forall a\in A)\}.$$
	Suppose that  $d$ is $\eta$-invariant, we observe that $d_H$ is $2^\eta$-invariant. For this take $g\in G$ and $A, B\in 2^X$ for which  $(g,A), (g,B)\in G*2^X$.  Let  $\epsilon>0$ with $A\subseteq N(B,\varepsilon)$ and  $B\subseteq N(A,\varepsilon)$. Now given $a\in A$ there exists  $b\in B$ such that $a\in B_d(b,\varepsilon)$,  then  $d(g\cdot a,g\cdot b)=d(a,b)<\varepsilon$ and we have proven than  $\eta_g( A)\subseteq N(\eta_g(B),\varepsilon),$ in a similar way one shows that  $\eta_g(B)\subseteq N(\eta_g( A),\varepsilon)$. Therefore, $d_H(\eta_g( A), \eta_g( B))\leq\varepsilon$,  and $d_H(\eta_g( A),\eta_g( B))\leq d_H(A,B)$. 
	
	On the other hand,  take $\varepsilon>0$ with  $\eta_g( A)\subseteq N(\eta_g( B),\varepsilon)$ and  $\eta_g( B)\subseteq N(\eta_g( A),\varepsilon)$. For  $a\in A$ choose  $b\in {B}$ such that $g\cdot a\in B_d(g\cdot b,\varepsilon)$ then  $d(a,b)=d(g\cdot a,g\cdot b)<\varepsilon$ and  $A\subseteq N(B,\varepsilon),$ again one verifies  $B\subseteq N(A,\varepsilon)$ which implies $d_H(A,B)\leq d_H(\eta_g( A),\eta_g( B))$ hence $d_H(A,B)=d_H(\eta_g( A),\eta_g( B)),$ as desired.
\end{exe}

 % That is  equivalente a que para cada $g\in G$  tal que $X_g\neq\emptyset$ la función $\eta_g: X_{g^{-1}}\rightarrow X_g$ sea una isometría. 
It follows from   \cite[Proposition 5]{A0} that     there is a compatible  $\eta$-invariant metric for  $X$ provided that  $\eta$ is global and $G$ is countably compact.  Our next goal is to obtain a generalization of this result to the frame of partial actions. %, but first we need a couple of results. %  in the case that $G$ is compact and metric.%In the next result we giveEl Tereoma \ref{propiedades} nos permite concluir lo mismo para acciones parciales con dominio cerrado y $G$ es métrico.

\begin{teo}\label{metrinv}
	Let $\eta:G*X\rightarrow X$ be a partial action,  then   $X$ and $X_G$ are  metrizable by a invariant metric under any of the following conditions:
\begin{itemize}
\item [(i)] $G$ is countably compact and $X_G$ is metrizable.
\item[(ii)] $G$ is compact and metric, $X$ is metric and $G*X$ is closed.
\end{itemize}
Moreover if $(i)$ holds and $X_G/\sim G$ is $T_1$, then  $X\!/\!\sim_G$ is metrizable.
 % there is a $\eta$-invariant metric for $X$ which is compatible with its topology.% $X$ admite una métrica invariante bajo $\eta$, que es compatible con su topología.
\end{teo}
\begin{proof} In both cases it is enough to prove that $X_G$ has a compatible $\mu$-invariant  metric  $\rho.$ Indeed,  since $\eta$ is continuous we have by \cite[Proposition 3.12]{KL}  that the spaces $X$ and $\iota(X)$  are homeomorphic, where $\iota$ is given by (\ref{iota}),  thus one obtains an invariant metric for $X$  by restricting $\rho$  to $\iota(X)$. (i) Since the action $\mu$ of  $G$ on $X_G$ given by \eqref{action} is continuous  the result follows from  \cite[Proposition 5]{A0}. (ii)  In this case the space $G\times X$ is metrizable, thus $X_G$ is metrizable  thanks to Theorem \ref{propiedades} and again the result follows from \cite[Proposition 5]{A0}. To show the last assertion, we observe that $X_G$ admits and invariant metric, then the result follows from \cite[Theorem 2.16]{SDN} and Lemma \ref{lemaaux}.
\end{proof}

\begin{rem}\label{nomet}It is known that when $G$ acts globally on a space $X$ admiting an invariant metric, then the space $X/\!\sim_G$ is metric provided that it is $T_1,$ however this not hold for partial actions. For a concrete example take the partial action given in Remark \ref{notrec} and use Theorem \ref{globor} and Remark \ref{cernes}.
\end{rem} The following result tells us that one needs to impose the regularity condition on   $X\!/\sim_G$.
\begin{pro}\label{orbmet} Let $X$ be a separable second countable space endowed with a partial action of $G,$ then the following assertions are equivalent.
\begin{enumerate}
\item[(i)] $X/\!\sim_G$ is metrizable;
\item [(ii)] $X/\!\sim_G$ is regular and $T_1.$
\end{enumerate}
\end{pro}
\begin{proof} Clearly (i) implies (ii). To see  that (ii) implies (i), notice that $X/\!\sim_G$ is separable and second countable, because the quotient map $\pi_G$ is open. Therefore, by Urysohn’s
	metrization Theorem, the space $X\!/\sim_G$ is metrizable.
\end{proof}
%If moreover $G$ is separable, we recover Theorem 4.1 from \cite{PU1}.

%
%\begin{rem}
%	Por el Ejemplo \ref{notclo}, en el Teorema \ref{metrinv} la hipótesis de que $G*X$ sea cerrado no es necesaria.
%\end{rem}

\section{Partial actions on orbit spaces}\label{parinq} Let $\eta$ be a partial action of $G$ on $X$  and $H$ be a normal subgroup of $G.$ 
The idea now is  to construct a partial action of $G/H$  on $X/\!\sim_H.$  If $\eta$ is a global action, then $G/H$ acts globally on $X/\!\sim_H$ via  
\begin{equation}\label{globquo}\eta_{gH}(H\cdot x)= H\cdot (g\cdot x),\end{equation} for any $g\in G$ and $x\in X.$

 For the case of partial action, we notice that  mimicking the construction above does not yield to a partial action of $G/H$  on $X/\!\sim_H$ because  it is not natural how to define the set $G/H*(X/\!\sim_H).$ Indeed the construction of such a partial action    is essentially more laborious than the global one,  as we shall see in the next.

%\begin{defi}
%	Let $\eta: G\ast X\longrightarrow X$ be a topological  partial action of a group $G$ on a set $X$. We say that $\eta$ is faithful if for each $(g,x)\in G\ast X$ such that $\eta(g,x)=x$, we have $g=1$.
%	\end{defi}

%\begin{rem} Let $\eta$ be a partial action of $G$ on a space $X,$ then $\eta$ is a free partial action if and only if its enveloping action $\mu$ acts freely on $X_G.$
%\end{rem}

%Now we give one of  the main results of this section.

\begin{teo}\label{lem4.3}
	Let $\eta$ be a   partial action of  $G$ on  $X$ and $H$ be a normal subgroup of $G$. Then there is a continuous partial action $\eta_{G/H}$ of $G/H$ on $X/\!\sim_H$, such that the the orbit spaces  $(X/\!\sim_H)/\!\sim_{G/H}$ and $X/\!\sim_G$ are homeomorphic.
\end{teo}
		
\begin{proof}
Let $\mu$ be  the globalization of $\eta$. Then $\mu$  is continuous and by \eqref{globquo}  it induces a continuous  action $\mu_{G/H}$ on $X_G/\!\sim_H$ as follows: 
	\begin{equation*}\label{mugh}\mu_{gH}: X_G/\!\sim_H\ni H[t,x] \mapsto H[gt,x]\in X_G/\!\sim_H,\end{equation*}
for each $gH\in G/H$. %Note that $\tau_{G/H}$ is free. In fact, let $g,t\in G$ and $x\in X$ be such that $\Pi_H([gt,x])=\Pi_H([t,x]).$ Then $[hgt,x]=[t,x]$ for some $h\in H$ and by \eqref{eqgl}  we have $x\in X_{(hgt)^{-1}t}$ and $\eta_{t^{-1}hgt}(x)=x$, and the fact that $\eta$ is free implies   $hg=1$ and $g\in H$. This shows that $gH=H$ and $\tau_{G/H}$ is free. 
Now let $\varphi$  be defined by \eqref{var}. By   Example \ref{indu} and Lemma \ref{lemaaux} the map   $\mu_{G/H}$ induces a continuous partial action   $\eta'_{G/H}$ of $G/H$ on  the open  set ${\rm Im}(\varphi)$ of  $X_G/\!\sim_H,$ where  $\eta'_{G/H}=\{\eta'_{gH}:X_{g\m H} \to X_{gH}\}_{gH\in G/H}$ and 
\begin{equation}\label{xgh}X_{gH}=\mu_{gH}({\rm Im}(\varphi))\cap {\rm Im}(\varphi)\,\,\,\text{ and  }\,\,\, \eta'_{gH}=\mu_{gH}\restriction{X_{g^{-1}H}}.\end{equation}   Let $\Omega:=X/\sim_H$, then one obtains a partial action   $\eta_{G/H}$ of $G/H$ on $\Omega$ by setting  
\begin{equation}\label{zgh}\Omega_{gH}=\varphi^{-1}(X_{gH}), g\in G\,\,\,\text{ and  }\,\,\,\eta_{gH}: \Omega_{g^{-1}H}\ni x\mapsto\varphi^{-1}(\eta'_{gH}(\varphi(x)))\in  \Omega_{gH}.\end{equation}  Then % On the other hand, using $\eta'_{G/H}$ and the  embedding $\varphi$  we  construct  the partial action $\eta_{G/H}$ of $G/H$ %=(\{\eta_{gH}\},\{T_{gH}\})$ of $G/H$

\begin{equation}\label{etaq}\eta_{gH}(x)=(\varphi^{-1}\circ \mu_{gH}\circ \varphi)(x),\end{equation} for each $x\in \Omega_{g\m H}$. The fact that   $\eta_{G/H}$ is continuous is straightforward.

Let $\sim_{G/H}$ be the orbit  equivalence relation in $\Omega$ induced by $\eta_{G/H}$. To finish the proof we show that  the spaces $\Omega/\!\sim_{G/H}$ and $X/\sim_G$ are homeomorphic.  Consider the diagram:
\begin{center}
	$\xymatrix{X \ar[d]_-{\pi_H} \ar[r]^-{\pi_G}& X\!/\sim_G \\
		\Omega \ar[r]_-{\pi_{G/H}}& \Omega/\!\sim_{G/H}\ar[u]^-{\psi}
	}$
\end{center}
where $\psi$ is made such that it commutes, that is  
\begin{equation}\label{homeo} \psi(\pi_{G/H}(\pi_H(x)))=\pi_G(x),\end{equation} for each $x\in X$. Let us prove that $\psi$ is well defined. Take  $z\in \Omega/\!\sim_{G/H}$ and $x,y\in X$ such that $\pi_{G/H}(\pi_H(x))=\pi_{G/H}(\pi_H(y))$. Then there is $g\in G$ with 
$$\pi_H(y)=\eta_{gH}(\pi_H(x)) \stackrel{\eqref{etaq} }=\varphi^{-1}(\mu_{gH}(\varphi(\pi_H(x))))=\varphi^{-1}(H[g,x]),$$  which implies  $H[g,x]=H[1,y]$ and there is $h\in H$ such that $[hg,x]=[1,y],$ thus  $\eta_{hg}(x)=y$ and $\pi_G(x)=\pi_G(y),$ which shows that $\psi$ is well defined. Moreover by its construction,  the map  $\psi$ is continuous and surjective.

Let us prove that $\psi$ is injective. Let $z_1,z_2\in \Omega/\!\sim_{G/H}$ be such that $\psi(z_1)=\psi(z_2)$, and let $x,y\in X$ be such that $\pi_{G/H}(\pi_H(x))=z_1$ and $\pi_{G/H}(\pi_H(y))=z_2$. Since $\pi_G(x)=\pi_G(y)$, there is $g\in G^x$ satisfying $\eta_g(x)=y$. To prove that $z_1=z_2$ we need  to find $t\in G$ for which $\eta_{tH}(\pi_H(x))=\pi_H(y)$. We claim that $\eta_{gH}(\pi_H(x))=\pi_H(y).$ In fact, by \eqref{etaq} we get \begin{center}
	$\eta_{gH}(\pi_H(x))=\varphi^{-1}(\mu_{gH}(\varphi(\pi_H(x))))=\varphi^{-1}(H[g,x]),$
\end{center}
and 
%\begin{center}
	$\varphi(\pi_H(y))=H[1,y]=H[g,x],$
%\end{center}
then $\eta_{gH}(\pi_H(x))=\pi_H(y)$ and  $\psi$ is injective.  Finally  let $U\subseteq \Omega/\!\sim_{G/H}$  be an open set. Since $\pi_G$ is open, $\pi_G(\pi_H^{-1}(\pi_{G/H}^{-1}(U)))\subseteq X/\!\sim_G$ is open. Thus $\psi(U)$ is open and $\psi: \Omega/\!\sim_{G/H}\to X_G/\!\sim_G $ is  a homeomorphism. \end{proof}
%{\color{red}

%We give the next

%\begin{exe} Of Theorem \ref{lem4.3}\end{exe}}

The following result describes explicitly the partial action $\eta_{G/H}$ and its globalization.

\begin{teo}\label{pro4.4}
Let $\eta$ be a   partial action of  $G$ on  $X$, $H$ be a normal subgroup of $G$ and  $\eta_{G/H}$ be the partial action of $G/H$ on $X/\!\sim_H$ defined by  \eqref{etaq} . Then the following assertions hold.
\begin{enumerate}
\item  [(i)]For $g\in G$ we have $(X/\!\sim_H)_{gH}=\{\pi_H(x)\mid  \exists h\in H \text{ such\ that}\ (hg\m,x)\in G*X \}$.
\item[(ii)] The domain of $\eta_{G/H}$ is 
$$G/H* X/\!\sim_H=\{(gH,\pi_H(x)): (g,x)\in G\times X\,\,\wedge\,\,\exists h\in H\text{ such\ that}\ (hg,x)\in G*X\}.$$
\item[(iii)] We have 
\begin{equation}\label{etaghh}\eta_{G/H}: G/H* X/\!\sim_H\ni (gH,\pi_H(x))\mapsto \pi_H((hg)\cdot x)\in X/\!\sim_H,\end{equation} where  $h\in H$ is  such\ that $(hg,x)\in G*X.$
\item [(iv)] The globalization of $\eta_{G/H}$ is  $(G/H)$-equivalent  to $X_G/\!\sim_H,$ where $G/H$ acts on $X_G/\!\sim_H,$ via $\mu_{G/H}.$
\end{enumerate}
\end{teo}
\begin{proof}  (i). Take $g\in G$ and $x\in X$ such that $\pi_H(x)\in (X/\!\sim_H)_{gH}$ then by \eqref{zgh} $\varphi(\pi_H(x))=H[1,x]\in X_{gH}$  and   \eqref{xgh}  gives an element $y\in X$ such that  $\mu_{gH}(H[1,y])=H[1,x],$ that is, $H[g,y]=H[1,x]$ and $[h_0,x]=[g,y]$ for some $h_0\in H$,  therefore $(g^{-1}h_0,x)\in G*X.$ Since $H$ is normal in $G$ we have $g^{-1}h_0=hg^{-1}$ for some $h\in H$ and  $(hg\m, x)\in G*X$. Conversely if $x\in X$ verifies $(h_0g\m, x)\in G*X$ for some $h_0\in H$. Then $h_0g\m=g\m h$ for some $h\in H$ and we have  $[h,x]=[g,y],$ where  $y=(g^{-1}h)\cdot x$  and
$$\varphi(\pi_H(x))=H[1,x]=H[1,(h\m g)\cdot y]\stackrel{\eqref{eqgl}}=H[h\m g, y]=H[ g, y]=\mu_{gH}(H[1,y])\in \mu_{gH}( {\rm Im}\varphi),$$ thus $\pi_H(x)\in (X/\!\sim_H)_{gH}$ thanks to equations \eqref{xgh} and \eqref{zgh}.

(ii).	This is a consequence  of part (i) and the fact that $(gH,\pi_H(x))\in G/H* X/\!\sim_H$ if and only if $\pi_H(x)\in (X/\!\sim_H)_{g\m H}$.

(iii). For $(gH,\pi_H(x))\in G/H* X/\!\sim_H,$ there exists  $h\in H$ such that $(hg,x)\in  G*X$. Then $[hg,x]=[1,(hg)\cdot x ] $  and $\varphi(\pi_H((hg)\cdot x))=H[hg,x]=H[g,x]$. Then follows by \eqref{etaq} that  $$\eta_{G/H}(gH,\pi_H(x))=\varphi^{-1}(H[g,x])=\pi_H((hg)\cdot x),$$ as desired.

(iv).  By Lemma \ref{lemaaux} we know that ${\rm Im}\varphi=\{H[1,x]\mid x\in X\},$ then $\mu_{G/H}[{\rm Im}\varphi]=X_G/\!\sim_H,$ thus by (ii) of Proposition \ref{kellen}  the spaces $({{\rm Im}\varphi})_{G/H}$  and $X_G/\!\sim_H $  are homeomorphic. % follows by \cite[Proposition 3.5]{KL} that the map $\alpha: {{\rm Im}\varphi}_G \ni [g, H[1,x]]\mapsto H[g,x]\in X_G/\!\sim_H $ is a  well defined bijection. To prove that it is a homeomorphism  notice that $\alpha\circ \lambda =\Pi_H\circ q_X$ where  $q_X$ and $q_{{\rm Im}\varphi}$ are given by \eqref{quo}, $\Pi_H:X_G\rightarrow X_G/\!\sim_H$  is the induced quotient map and $\lambda =q_{{\rm Im}\varphi}\circ ({\rm id}_G, \varphi\circ \pi_H)$ then $\alpha\circ \lambda$ is continuous and open because $\Pi_H$ and $q_X$ are, and the fact that $\lambda$ is a continuous open surjection implies that $\alpha$ is continuous and open, thus a homeomorphism. 
Now we must show that the spaces ${\rm Im}\varphi$ and $X/\!\sim_H$ are $G/H$-equivalent,  but by  (i) in Lemma  \ref{kellen} it is enough to show that $\varphi$ is a $(G/H)$-map, and this follows from \eqref{zgh}. 
\end{proof}

\begin{exe}  Consider the partial action  $\eta:\mathbb{Z}*X\rightarrow X$ of Example \ref{componer} and let $m\in \mathbb{Z}^+$ be such that $f^m={\rm id}_U$, where $m$ is the smallest positive  integer with this property. If $H=m\mathbb{Z}$, then  the induced quotient morphism $\pi_H$ satisfies $\pi_H(x)=\{x\},$ for any $x\in X,$ thus the   spaces $X$ and $X/\!\sim_H$ are homeomorphic.  Now  we  determine $\eta_{\mathbb{Z}/H}.$ Take $(k+H,\pi_H(x))\in \mathbb{Z}/H* X/\!\sim_H,$ if $k\in H$,  by \eqref{etaghh}  we get $$\eta_{\mathbb{Z}/H}(k+H,\pi_H(x))=\eta_{\mathbb{Z}/H}(H,\pi_H(x))=\pi_H(x).$$
	% Tenemos que el subgrupo $H=m\mathbb{Z}$ de $\mathbb{Z}$ actúa en $X$ por restric
Suppose  $k\notin H.$ By (ii) of Theorem \ref{pro4.4}, there is $h\in H$ such that $(h+k,x)\in \mathbb{Z}*X$ and $\eta_{\mathbb{Z}/H}(k+H,\pi_H(x))=\pi_H(\eta(h+k,x)),$ thanks to \eqref{etaghh}. Since $(h+k,x)\in \mathbb{Z}*X$ and $k\notin H$ the equality \eqref{zx} implies  $x\in U$. Then, $(h,x)$  and $(k+h,x)$ belong to  $\mathbb{Z}*X,$ which gives % De hecho, como $h\cdot x=x$, entonces $(k,h\cdot x)\in G*X$ y por definición de acción parcial, tenemos que $(k+h)\cdot x= k\cdot (h\cdot x)= k\cdot x$. En otras palabras, 
$\eta(h+k,x)=\eta(k+h,x)=f^{k+h}(x)=f^k(x)$. We have shown that if $k\notin H$ with $(k+H,\pi_H(x))\in \mathbb{Z}/H* X/\sim_H$,  one gets  $$\eta_{\mathbb{Z}/H}(k+H,\pi_H(x))=\pi_H(\eta(h+k,x))=\pi_H(f^k(x))=\pi_H(\eta(k,x)).$$

\end{exe}

\begin{coro}
	Let $G$ be compact and Hausdorff group, $H$ be a closed normal subgroup of   $G,$  and  $\eta: G*X\rightarrow X$ be a continuous partial action on a compact Hausdorff space $X$. If $G*X$ is closed in $G\times X$, then $G/H* X/\!\sim_H$ is closed in $G/H\times X/\!\sim_H$.
\end{coro}
\begin{proof} Let $\eta'$ be the partial action defined  \eqref{xgh}. %in the proof   of Theorem \ref{lem4.3}, 
	By construction we get that $\eta_{G/H}$ and $\eta'_{G/H}$ are $G/H$-equivalent,  and  thus  %by \textcolor{blue}{ (ii) in Lemma \ref{lemaaux} (El lema anterior no tiene parte $(ii)$) }   
it is enough to show that
	$G/H*{\rm Im}(\varphi)$ is closed in $G/H\times {\rm Im}(\varphi)$. Having at hand Remark \ref{domiclopen} we only need to see that  ${\rm Im}(\varphi)$ is closed in $X_G/\!\sim_H.$ Now by  (ii) in Theorem \ref{propiedades} the enveloping space $X_G$ is Hausdorff and  since   $H$ is compact then  the first item of Theorem \ref{propiedades} implies that  $X_G/\!\sim_H$ is Hausdorff.  Also $X/\!\sim_H $ is compact which implies that  $\varphi$ is a closed map, then ${\rm Im}(\varphi)$ is closed in  $X_G/\!\sim_H$ and  $G/H*{\rm Im}(\varphi)$ is closed  in $G/H\times {\rm Im}(\varphi)$ which finishes the proof.
\end{proof}

The following is clear.
\begin{lem}\label{induci}
	Let  $G$ and  $H$ be topological groups and $\phi: G\rightarrow H$ be a group homomorphism. If $\{\eta_h\colon X_{h\m}\to X_h \}_{h\in H} $ is a   partial action of $H$ on  $X,$ then  the family  $\{\eta_{\phi(g)}\colon U_{g\m}\to U_{g}\}_{g\in G} ,$ where $U_g=X_{\phi(g)}, g\in G,$ is a   partial action of $G$ on  $X$ such that   %$\eta: H*X\rightarrow X$ y $\phi\times id_X: G\times X\rightarrow H\times X$, entonces definimos 
	\begin{equation}\label{otrain}G*X= (\phi\times {\rm id}_X)^{-1}(H*X)\,\,\,{\rm and}\,\,\,  G*X\ni (g,x)\mapsto \eta(\phi(g),x)\in X.\end{equation} 
\end{lem}

\begin{rem}\label{dominio}

 Using  $\eta_{G/H}$ and the  canonical homomorphism $p_H: G\rightarrow G/H,$ it follows by Theorem \ref{lem4.3} and  Lemma \ref{induci}  that  there is a partial action $\eta^{p_H}$ of $G$ on  $X/\!\sim_H$ which by \eqref{otrain} has domain
  \begin{equation}\label{dometa}G* (X/\!\sim_H)=\{(g,\pi_H(x))\mid g\in G,\ x\in X,\ (gH,\pi_H(x))\in G/H* X/\!\sim_H\},\end{equation} and 
  \begin{equation}\label{etaqu}\eta^{p_H}(g, \pi_H(x))= \eta_{G/H} (gH, \pi_H(x)).%\stackrel{\eqref{etaq}}=\varphi^{-1}(H[g,x]).
\end{equation} From now on we always consider $G$ acting partially on $X/\!\sim_H$  via  $\eta^{p_H}.$

%$\eta_{gH}(H^x\cdot x)=\varphi^{-1}( H[g,x])=(\varphi^{-1}\circ \tau_{gH}(H[1,x])=.$$
\end{rem}

%\footnote{Describir explicitamente esa acci\'on parcial  cociente}

Let $H_1, H_2$ be subgroups of $G$ such that $H_1\subseteq H_2$. We define $\pi_{H_1, H_2}: X/\!\sim_{H_1}\to X/\!\sim_{H_2}$ as the only map  such that  
\begin{equation}\label{pih12}\pi_{H_2}= \pi_{H_1, H_2}\circ \pi_{H_1},\end{equation} in particular  for a subgroup $H$ of $G$ the map $\pi_{H,H}$ is the identity on  $X/\!\sim_{H}.$
\begin{pro}\label{equivari}
	Let $H, H_1$ and $H_2$ be normal subgroups of $G$ with $H_1\subseteq H_2$. Then  $\pi_H$ and $\pi_{H_1,H_2}$ are $G$-maps.
\end{pro}
\begin{proof} We first show that $\pi_H$ is a $G$-map. Take $(g,x)\in G*X,$ by (ii) of Theorem  \eqref{pro4.4} the pair $(gH,\pi_H(x))$ belongs to  $G/H* X/\!\sim_H$ and follows by \eqref{etaghh}  that  $\pi_H(\eta(g, x))=\eta_{G/H}(gH,\pi_H(x))$.  Hence $(g,\pi_H(x))\in G* X/\!\sim_H$ and  $\eta^{p_H}(g,\pi_H(x))=\pi_H(\eta(g, x)) $ which shows that $\pi_H$ is a $G$-map.  Now we show that $\pi_{H_1, H_2}$ is a $G$-map. Suppose  $(g,\pi_{H_1}(x))\in G* X/\!\sim_{H_1}$.  We need to show that   $(g,\pi_{H_2}(x))\in G* X/\!\sim_{H_2}$ and $\pi_{H_1,H_2}(\eta_{G/H_1}( gH_1, \pi_{H_1}(x)))=\eta_{G/H_2}(gH_2,\pi_{H_2}(x)).$ We have $(gH_1,\pi_{H_1}(x))\in G/H_1* X/\!\sim_{H_1}$ using  (ii) of  Theorem \ref{pro4.4} there exists an $h\in H_1\subseteq H_2$ such that  $(hg,x)\in G*X$, thus  $(gH_2,\pi_{H_2}(x))\in G/H_2* X/\sim_{H_2}$ and $(g,\pi_{H_2}(x))\in G*X/\sim_{H_2}$. It follows from \eqref{etaghh} that % and the fact that $\pi_{H_1}$	is $G$-invariant that Por 
 $$\eta^{p_{H_1}}(g, \pi_{H_1}(x))=\eta_{G/H_1}(gH_1,\pi_{H_1}(x))=\pi_{H_1}(\eta(hg, x)),$$ in a similar way  $\eta^{p_{H_2}}(g, \pi_{H_2}(x))=\eta_{G/H_2}(gH_2, \pi_{H_2}(x))=\pi_{H_2}(\eta(hg, x))$.  Therefore
$$\pi_{H_1,H_2}(g\cdot \pi_{H_1}(x))=\pi_{H_1,H_2}(\pi_{H_1}(hg\cdot x))=\pi_{H_2}(hg\cdot x)=g\cdot\pi_{H_2}(x),$$ 
and we conclude that $\pi_{H_1,H_2}$ is a  $G$-map. 
\end{proof}

\subsection{Inverse limits} 

As an application of Theorem \ref{lem4.3} we extend \cite[Theorem 9]{A0} to the context of partial actions. Suppose that $G$ is a compact group, let $(I,\leq)$ be a directed set and consider an inverse system  $\{G_i; p_i^j; I\}$  in the category of topological groups such that $G =\displaystyle{ \lim_{\longleftarrow}}\ G_i$, where  $\{p_i: G\rightarrow G_i\}_{i\in I}$ is the family of projections  such that $p_i^j\circ p_j=p_i$   for  $i ,j\in I$ and $i\leq j.$ 
Take $i\in I,$  then  $H_i=\ker(p_i)=p_i^{-1}(\{e_i\})$ is a closed normal  subgroup of  $G$ thus  is compact and $H_j\leq H_i$  for every $i, j\in I$ with $i\leq j$.  Let $\eta$  be a partial action of   $G$ on $X,$ now for  $i\in I$  the group $H_i$  acts partially on $X$ via restriction, setting $X_i=X/\sim_{H_i}$ we denote by  $\pi_i^j=\pi_{H_j,H_i}: X_j\rightarrow X_i$,  the $G$-map defined in \eqref{pih12} and $\pi_i=\pi_{H_i}:X\to X_i,$ the orbit equivalence map.

We proceed with the next.

\begin{lem}\label{lema 3.1}
	Following the notations above consider   $i,j\in I$ with $i\leq j$ let $\eta:G*X\to X$ be a continuous partial action with $G*X$ is closed, then the family  $\{\pi_i:X\rightarrow X_i\}_{i\in I}$ separates points of closed sets in $X$.
%	\begin{enumerate}
%		\item [(i)] $\pi_i$ is a  $G$-map,
%		\item [(ii)] $\pi_i^j\circ \pi_j=\pi_i$,
%		\item [(iii)] 
%	\end{enumerate}
\end{lem}
\begin{proof} The proof follows the lines of \cite[Lemma 3]{A0}, where it is shown that $\pi_i(x)\notin \pi_i(C)$ for any  $x\in X$ and  $C\subseteq X$  a closed subset such that $x\notin C.$ On the other hand, the fact that $G*X$  is closed is used to  guarantee that  $H_i*X=(G\ast X)\cap (H_i\times X),$  is closed in $H_i\times X,$ which implies that $\pi_i$ is closed, for any $i\in I.$    Then the family $\{\pi_i:X\rightarrow X_i\}_{i\in I}$ separates points of closed sets in $X,$ as desired.% $\pi_i(x)\notin\overline{\pi_i(C)}$. \textcolor{blue}{Then} the family $\{\pi_i:X\rightarrow X_i\}_{i\in I}$ separates points of closed sets in $X$. \\\\
	%Item (i) was already proved in Proposition \ref{equivari} while (ii) is clear.  Now we check (iii),
	% Take $x\in X$ and let $C\subseteq X$ be a closed subset such that $x\notin C$. Consider the continuous map $\eta^x:G^x\ni g\mapsto \eta(g,x)\in X$. Since $x=\eta^x(1)\in X\setminus C$ there is an open set $U$ of $G^x$ such that $1\in U$ and  $\eta^x(U)\subseteq X\setminus C$. Let $W\subseteq G$ be open with $U=W\cap G^x.$ Since  $G =\displaystyle{ \lim_{\longleftarrow}}\ G_i$ there are  $i\in I$ and $V_i\subseteq G_i$ open for which $1\in p_i^{-1}(V_i)\subseteq W$, from this we have $1\in  p_i^{-1}(V_i)\cap G^x\subseteq U$. Moreover,  $e_i=p_i(1)\in V_i$ then  $H_i\cap G^x=p_i^{-1}(\{e_i\})\cap G^x\subseteq p_i^{-1}(V_i)\cap G^x\subseteq U$ which gives $H_i^x\cdot x=\theta^x(H_i\cap G^x)\subseteq X\setminus C.$ We have shown that  $\pi_i(x)\notin \pi_i(C)$, but \textcolor{blue}{$G\ast X$ is closed and $H_i*X=(G\ast X)\cap (H_i\times X),$ therefore} $H_i*X$ is closed in $H_i\times X$, then $\pi_i$ is closed and $\pi_i(x)\notin\overline{\pi_i(C)}$. \textcolor{blue}{Then} the family $\{\pi_i:X\rightarrow X_i\}_{i\in I}$ separates points of closed sets in $X$.
\end{proof}

 %\textcolor{blue}{Keeping the notations above %$G/H_i$ acts partially on $X_i$ by Theorem  \ref{pro4.4}  %donde $\pi_i: X\rightarrow X_i$ es la proyección orbital, para cada $i\in I$$G$ acts partially on  $X_i$ by   Lemma \ref{induci}.} 
Assuming $X$  Hausdorff and letting  $i,j,k\in I$ be such that $i\leq j\leq k.$ For  $x\in X,$  we have  $\pi_i^k(H_k^x\cdot x)=(\pi_i^j\circ \pi_j^k)(H_k^x\cdot x),$ and  %además, por la Proposición \ref{equivari} y el Remark \ref{induci}, se deduce que 
$\mathcal{X}=\lbrace X_i, \pi_i^j, I\rbrace$ is an inverse system of  spaces endowed with partial actions of $G.$

We finish this work with the next.
\begin{pro}\label{parinv}
Under the assumptions above,  let $\mathcal X=\lbrace \varphi_i:\displaystyle{\lim_{\longleftarrow}}X_i\rightarrow X_i\rbrace_{i\in I}$ be the family of projections associated  to  $\displaystyle{\lim_{\longleftarrow}}X_i,$ if $X$ is Hausdorff and $G\ast X$ is closed in $G\times X,$  then the following assertions hold.
\begin{itemize}
	\item  [(i)] There is a partial action $\theta$  of $G$ on $\displaystyle{\lim_{\longleftarrow}}X_i$ such that $X$ is $G$-equivalent to $\displaystyle{\lim_{\longleftarrow}}X_i.$
	\item  [(ii)]  For any $j\in I$ the  diagram   \begin{center}
		$\xymatrix{G*{\lim\limits_{\longleftarrow}}X_i \ar[d]_-{{\rm id}_G \times \varphi_j} \ar[r]^-{\theta}&\ar[d]^-{\varphi_j} {\lim\limits_{\longleftarrow}}X_i \\
			G*X_j\ar[r]_-{\eta^{p_{H_j}}}& X_j
		}$ 
	\end{center}  commutes, where $\eta^{p_{H_j}}$ is  the partial action of $G$ on $X_j$  given by \eqref{etaqu}.
\end{itemize} %and any  $H_i$ is compact, 

\end{pro}
\begin{proof} (i)
%Sea $\displaystyle{\lim_{\longleftarrow}}X_i$ el límite del sistema $\mathcal{X},$ 
 % $\mathcal{X}=\lbrace \varphi_i:\displaystyle{\lim_{\longleftarrow}}X_i\rightarrow X_i\rbrace_{i\in I}$ be the family of projections associated  to  $\displaystyle{\lim_{\longleftarrow}}X_i.$
 It is not difficult to see that  the family $\Pi=\lbrace \pi_i:X\rightarrow X_i\rbrace_{i\in I}$ is compatible with  $\mathcal{X}$  then by the universal property of the inverse limit there exists a continuous map  $\lambda:X\rightarrow\displaystyle{\lim_{\longleftarrow}}X_i,$ such that $\varphi_i\circ \lambda=\pi_i,$ for any $i\in I.$  We shall prove that $\lambda$ is a homeomorphism.  First,   
by   Lemma \ref{lema 3.1}, the family $\Pi$ separates points of closed sets in $X,$ further by (i) in Theorem \ref{propiedades} each orbit space $X_i$ is $T_2,$ then the  map $\lambda$ is an embedding.   Let $(x_i)_{i\in I}\in \displaystyle{\lim_{\longleftarrow}}X_i,$ since $H_i\ast X$ is closed in $H_i\times X$ and by Lemma \ref{cerrada}  the map $\pi_i$ is perfect, we have $A_i=\pi^{-1}_i(x_i)$ is a  compact subset of $X$, for all $i\in I$. Now write  $\mathcal{A}=\lbrace A_i\rbrace_{i\in I}$ and take  $i,j\in I$ such that $i\leq j.$ For  $y\in A_j$ we have  $\pi_i(y)=\pi_i^j(\pi_j(y))=\pi_i^j(x_j)=x_i,$ and  $A_j\subseteq A_i,$ from this one concludes that  $\mathcal{A}$ has the finite intersection property, therefore  $\bigcap\limits_{i\in I}A_i\neq \emptyset.$  Finally if  $y\in\bigcap\limits_{i\in I}A_i,$ then  $\pi_i(y)=x_i,$ that is  $(x_i)_{i\in I}=\lambda(y),$ therefore $\displaystyle{\lim_{\longleftarrow}}X_i=\lambda(X)$ and $\lambda$ is a homeomorphism. To define a partial action of $G$ on  $\displaystyle{\lim_{\longleftarrow}}X_i$ we set 
$G*\displaystyle{\lim_{\longleftarrow}}X_i=\left\{(g,x)\in G\times {\lim_{\longleftarrow}}X_i \mid (g,\lambda\m(x))\in G*X\right\}$ and  
$$\theta: G* \displaystyle{\lim_{\longleftarrow}}X_i\ni (g,x)\mapsto \lambda(g\cdot \lambda\m(x))\in  \displaystyle{\lim_{\longleftarrow}}X_i,$$ thus $\lambda$ is a $G-$map  which shows the first item. 

(ii) Take $j\in I$ we first check that the map ${\rm id}_G\times \varphi_j$ is well defined, that is for $(g, x) \in G* \lim\limits_{\longleftarrow}X_i$ 
one has that $(g, x_j)\in G*X_j,$  where $x=(x_i)_{i\in I}.$ Indeed, if $(g, x) \in G* \lim\limits_{\longleftarrow}X_i$ we get that $(g,\lambda\m(x))\in G*X$ which by item (ii) in 
Theorem \ref{pro4.4} implies $(gH_j, \pi_j(\lambda\m(x)))\in G/H_{j}* X_j$ and thus  $(g, x_j)=(g, \pi_j(\lambda\m(x)))\in G* X_j$ thanks to \eqref{dometa}, and  ${\rm id}_G\times \varphi_j$ is well defined. To check that the diagram commutes observe that 
$$\eta^{p_{H_j}}\circ ({\rm id}_G \times \varphi_j) (g, x)=\eta_{G/H_j}(gH_j, \pi_j(\lambda\m(x)))=\pi_j((hg)\cdot \lambda\m(x)),$$ where by (ii) of Theorem \ref{pro4.4}  the element $h\in H_j$ is  such\ that $(hg,\lambda\m (x))\in G*X.$ Since $\lambda\m (x) \in X_{g\m h\m}\cap X_{g\m}$ we get by item (ii) of Proposition \ref{QR} that $g\cdot \lambda\m(x)\in X_{h\m}$ thus   $(hg)\cdot \lambda\m(x)=h\cdot (g\cdot \lambda\m(x)) $ and $\pi_j((hg)\cdot \lambda\m(x))=\pi_j(g\cdot \lambda\m(x)).$ On the other hand $\varphi_j\circ \theta(g,x)=\varphi_j\lambda(g\cdot \lambda\m(x))=\pi_j(g\cdot \lambda\m(x)).$ Then $\eta^{p_{H_j}}\circ ({\rm id}_G \times \varphi_j) =\varphi_j\circ \theta$  which ends the proof.
\end{proof}

%
%\subsection{The universal Hausdorff Group} Given a topological group $G$ and $E=\overline{\left\{1\right\}},$ then $E$ is a closed normal subgroup of $G$ and the quotient $G/E$  is called the {\it universal de Hausdorff Group.}  This name is due to the  property that for any continuous group  homomorphism $f:G\to H$ where $H$ is Hausdorff there is a unique  continuous group homomorphism $f^*:G/E\to H$ such that $f=f^*v$ where $v:G\to G/E$ is the quotient map(see \cite[Excercises (i) p.34]{PH})

%	
%		\section{preguntas/pendientes}
%\begin{enumerate}
%\item articulo suecos outer
%
%
%\item categor\'ia reflexiva.
%\item Relacionar con partial Effros?
%\end{enumerate}
%		

	\end{document}